\documentclass[12pt]{amsart}
\usepackage{amsmath,amsthm,amsfonts,amscd,amssymb,eucal,latexsym,mathrsfs, appendix, yhmath, setspace}
\usepackage{mathrsfs}
\usepackage[numbers,sort&compress]{natbib}
\usepackage[all,cmtip]{xy}
\usepackage{enumerate}

\newtheorem{theorem}{Theorem}[section]

\newtheorem{proposition}[theorem]{Proposition}

\theoremstyle{definition}
\newtheorem{definition}[theorem]{Definition}

\newtheorem{example}[theorem]{Example}

\allowdisplaybreaks
\begin{document}

\title{Mixing unilateral backward shifts}

\author{Zhen Rong}

\address{\hskip-\parindent
Z.R., College of statistics and mathematics, Inner Mongolia University of Finance and Economics,
Hohhot 010000, China.}
\email{rongzhenboshi@sina.com}

\date{June 8, 2023}


\begin{abstract}
In this paper, we characterize mixing unilateral backward shifts on barrelled and ultrabarrelled topological sequence spaces respectively, which extend several well-known results in the existing literature. We present some nontrivial examples to show the validity of our results.
\end{abstract}

\maketitle

\section{Introduction} \label{S-introduction}
Throughout this article, let $\mathbb{N}$ denote the set of nonnegative integers and let $\mathbb{N}^{\ast}$ denote the set of positive integers. Let $\mathbb{Z}$ denote the set of integers. Let $\mathbb{K}$ denote the real number field $\mathbb{R}$ or the complex number field $\mathbb{C}$. Let $\mathbb{K}^{\mathbb{N}^{\ast}}$ denote the set of maps from $\mathbb{N}^{\ast}$ into $\mathbb{K}$. For each $n\in\mathbb{N}^{\ast}$, let $e_{n}$ denote the sequence that has 1 as its $n^{th}$ term and all other terms 0.

A subset $A$ of a vector space $X$ is called {\it convex} if $sA+tA\subseteq A$ for $0\leqslant t\leqslant1,s+t=1$; {\it balanced} if $tA\subseteq A$ for $|t|\leqslant1$; {\it absolutely convex} if $A$ is both balanced and convex; {\it absorbing} if for every $x\in X$ there exists $\varepsilon>0$ such that $tx\in A$ for $|t|<\varepsilon$.

Suppose that $X$ is a vector space over the scalar field $\mathbb{K}=\mathbb{R}$ or $\mathbb{C}$ with a topology $\mathfrak{T}$ such that the addition operation and scalar multiplication operation are continuous. Then the ordered pair $(X,\mathfrak{T})$ is called a {\it topological vector space} and the topology $\mathfrak{T}$ is called a {\it vector topology} for $X$.

A topological vector space $X$ is called a {\it locally convex space} if every neighborhood of 0 contains a convex neighborhood of 0; it is called an {\it F-space} if the topology of $X$ can be induced by some complete metric; and it is called a {\it Fr$\acute{e}$chet space} if $X$ is a locally convex F-space. A topological vector space $X$ is called a {\it topological sequence space} if $X$ is a linear subspace of $\mathbb{K}^{\mathbb{N}^{\ast}}$ and the convergence in $X$ implies coordinatewise convergence.

A subset $A$ of a topological vector space $X$ is called a {\it barrel} if $A$ is an absolutely convex absorbing closed set. A {\it barrelled} space is a locally convex space in which every barrel is a neighborhood of 0. It is known that a locally convex space which is of second category is barrelled (see \cite[page 136, Example 9.3.2]{Wilansky78}). In particular, every Fr$\acute{e}$chet space is barrelled.

Let $\mathfrak{T}, \mathfrak{T}^{'}$ be vector topologies for a vector space $X$. We say $\mathfrak{T}^{'}$ is $F$ {\it linked to} $\mathfrak{T}$ if $\mathfrak{T}^{'}$ has a local base of neighborhoods of 0, each of which is $\mathfrak{T}$ closed. A topological vector space $(X,\mathfrak{T})$ is called {\it ultrabarrelled} if $\mathfrak{T}$ is larger than every vector topology which is $F$ linked to $\mathfrak{T}$. It is known that every second-category topological vector space is ultrabarrelled (see \cite[page 10]{Waelbroeck}). In particular, every F-space is ultrabarrelled. In addition, there exists a barrelled space which is not ultrabarrelled. Also there exists an ultrabarrelled space which is not barrelled.

A sequence $(\mu_{n})_{n=1}^{\infty}$ in a topological vector space $X$ is called a {\it basis} if every $\mu\in X$ has a unique representation
$$\mu=\sum_{n=1}^{\infty}\alpha_{n}\mu_{n}$$
with scalars $\alpha_{n}\in\mathbb{K},n\geqslant1$.

A continuous linear operator $S$ on a topological vector space $X$ is called {\it hypercyclic} if there is an element $\mu$ in $X$ whose orbit $\{S^{n}\mu:n\in\mathbb{N}\}$ is dense in $X$; and {\it mixing} if for any pair $U,V$ of nonempty open subsets of $X$, there exists some nonnegative integer $N$ such that $S^{n}(U)\cap V\neq\emptyset$ for all $n\geqslant N$.

The historical interest in hypercyclicity is related to the invariant subset problem. The invariant subset problem, which is open to this day, asks whether every continuous linear operator on any infinite dimensional separable Hilbert space possesses an invariant closed subset other than the trivial ones given by $\{0\}$ and the whole space. After a simple observation, a continuous linear operator $S$ on a Banach space $X$ has no nontrival invariant closed subsets if and only if every nonzero vector $\mu$ is hypercyclic (i.e. the orbit $\{S^{n}\mu:n\in\mathbb{N}\}$ under $S$ is dense in $X$).

The best known examples of hypercyclic operators are due to Birkhoff \cite{Birkhoff}, MacLane \cite{MacLane} and Rolewicz \cite{Rolewicz}. Birkhoff \cite{Birkhoff} proved that the translation operator on the space $H(\mathbb{C})$ of entire functions is hypercyclic. MacLane \cite{MacLane} proved that the differentiation operator on $H(\mathbb{C})$ is hypercyclic. Hypercyclic shift operators on $l^{p}(1\leqslant p<+\infty)$ or $c_{0}$ were first studied by Rolewicz \cite{Rolewicz}. Since then the hypercyclic, mixing, weakly mixing and chaotic shift operators have been extensively studied (see \cite{Beauzamy,Bernal-Gonzalez,Bes-Peris,Birkhoff,Chan,Costakis-Sambarino,Gethner-Shapiro,
Godefroy-Shapiro,Grosse-Erdmann,Grosse-Erdmann-Peris,Gulisashvili-MacCluer,Kitai,MacLane,Martinez-Peris,Mathew,Rolewicz,Salas91,Salas95,Rong}).

Shift operators play an important role during the study of linear dynamical systems. Shift operators are of interest because many classical operators (such as the differentiation operator) can be viewed as such operators.

Costakis and Sambarino \cite{Costakis-Sambarino} characterized mixing unilateral and bilateral weighted backward shifts on $l^{2}$ and $l^{2}(\mathbb{Z})$. Recently Grosse-Erdmann and Peris \cite[pages 89-104]{Grosse-Erdmann-Peris} characterized mixing unilateral and bilateral weighted backward shifts on Fr$\acute{e}$chet sequence spaces.

In the case of mixing unweighted unilateral backward shifts on Fr$\acute{e}$chet sequence spaces, Grosse-Erdmann and Peris \cite[Theorem 4.5]{Grosse-Erdmann-Peris} proved the following results:

\begin{theorem}[\cite{Grosse-Erdmann-Peris}]
Let $X$ be a Fr$\acute{e}$chet sequence space in which $(e_{n})_{n=1}^{\infty}$ form a basis. Suppose that the unilateral unweighted backward shift $T$ is a linear operator on $X$, where
$$T(u_{1},u_{2},u_{3},\cdots)=(u_{2},u_{3},u_{4},\cdots).$$
Then $T$ is mixing if and only if $\lim\limits_{n\rightarrow\infty}e_{n}=0$.
\end{theorem}

In the case of mixing weighted unilateral backward shifts on Fr$\acute{e}$chet sequence spaces, Grosse-Erdmann and Peris \cite[Theorem 4.8 (b)]{Grosse-Erdmann-Peris} proved the following results:

\begin{theorem}[\cite{Grosse-Erdmann-Peris}]
Let $X$ be a Fr$\acute{e}$chet sequence space in which $(e_{n})_{n=1}^{\infty}$ form a basis. Let $\nu=(v_{n})_{n=1}^{\infty}$ be a sequence of nonzero scalars. Suppose that the unilateral weighted backward shift $T_{\nu}$ is a linear operator on $X$, where
$$T_{\nu}(u_{1},u_{2},u_{3},\cdots)=(v_{2}u_{2},v_{3}u_{3},v_{4}u_{4},\cdots).$$
Then $T_{\nu}$ is mixing if and only if $\lim\limits_{n\rightarrow\infty}(\prod\limits_{i=1}^{n}v_{i})^{-1}e_{n}=0$.
\end{theorem}

In this paper we will characterize mixing unilateral unweighted backward shifts on barrelled and ultrabarrelled topological sequence spaces respectively. Then we extend the corresponding results immediately to all weighted unilateral backward shifts on barrelled and ultrabarrelled topological sequence spaces via a simple conjugacy.

In the case of mixing unweighted unilateral backward shifts on barrelled topological sequence spaces, we prove the following results:

\begin{theorem}
Let $X$ be a barrelled topological sequence space in which $(e_{n})_{n=1}^{\infty}$ form a basis. Suppose that the unweighted unilateral backward shift $T$ is a continuous linear operator on $X$, where
$$T(u_{1},u_{2},u_{3},\cdots)=(u_{2},u_{3},u_{4},\cdots).$$
Then $T$ is mixing if and only if $\lim\limits_{n\rightarrow\infty}e_{n}=0$.
\end{theorem}

Grosse-Erdmann and Peris \cite[Theorem 4.5]{Grosse-Erdmann-Peris} proved Theorem 1.3 in the case of Fr$\acute{e}$chet sequence spaces, thus Theorem 1.3 generalizes Theorem 1.1.

In the case of mixing weighted unilateral backward shifts on barrelled topological sequence spaces, we prove the following results:

\begin{theorem}
Let $X$ be a barrelled topological sequence space in which $(e_{n})_{n=1}^{\infty}$ form a basis. Let $\nu=(v_{n})_{n=1}^{\infty}$ be a sequence of nonzero scalars. Suppose that the unilateral weighted backward shift $T_{\nu}$ is a continuous linear operator on $X$, where
$$T_{\nu}(u_{1},u_{2},u_{3},\cdots)=(v_{2}u_{2},v_{3}u_{3},v_{4}u_{4},\cdots).$$
Then $T_{\nu}$ is mixing if and only if $\lim\limits_{n\rightarrow\infty}(\prod\limits_{i=1}^{n}v_{i})^{-1}e_{n}=0$.
\end{theorem}

Grosse-Erdmann and Peris \cite[Theorem 4.8 (b)]{Grosse-Erdmann-Peris} proved Theorem 1.4 in the case of Fr$\acute{e}$chet sequence spaces, thus Theorem 1.4 generalizes Theorem 1.2.

In the case of mixing unweighted unilateral backward shifts on ultrabarrelled topological sequence spaces, we prove the following results:

\begin{theorem}
Let $X$ be an ultrabarrelled topological sequence space in which $(e_{n})_{n=1}^{\infty}$ form a basis. Suppose that the unweighted unilateral backward shift $T$ is a continuous linear operator on $X$, where
$$T(u_{1},u_{2},u_{3},\cdots)=(u_{2},u_{3},u_{4},\cdots).$$
Then $T$ is mixing if and only if $\lim\limits_{n\rightarrow\infty}e_{n}=0$.
\end{theorem}

Grosse-Erdmann and Peris \cite[Theorem 4.5]{Grosse-Erdmann-Peris} proved Theorem 1.5 in the case of Fr$\acute{e}$chet sequence spaces, thus Theorem 1.5 generalizes Theorem 1.1.

In the case of mixing weighted unilateral backward shifts on ultrabarrelled topological sequence spaces, we prove the following results:

\begin{theorem}
Let $X$ be an ultrabarrelled topological sequence space in which $(e_{n})_{n=1}^{\infty}$ form a basis. Let $\nu=(v_{n})_{n=1}^{\infty}$ be a sequence of nonzero scalars. Suppose that the unilateral weighted backward shift $T_{\nu}$ is a continuous linear operator on $X$, where
$$T_{\nu}(u_{1},u_{2},u_{3},\cdots)=(v_{2}u_{2},v_{3}u_{3},v_{4}u_{4},\cdots).$$
Then $T_{\nu}$ is mixing if and only if $\lim\limits_{n\rightarrow\infty}(\prod\limits_{i=1}^{n}v_{i})^{-1}e_{n}=0$.
\end{theorem}

Grosse-Erdmann and Peris \cite[Theorem 4.8 (b)]{Grosse-Erdmann-Peris} proved Theorem 1.6 in the case of Fr$\acute{e}$chet sequence spaces, thus Theorem 1.6 generalizes Theorem 1.2.

This paper is organized as follows. In Section~\ref{S-mixing barrelled} we characterize mixing unilateral backward shifts on barrelled topological sequence spaces. Furthermore, we exhibit several mixing and non-mixing unilateral backward shifts on some non-metrizable topological sequence spaces in Section 2. These examples show that our results are valid. In Section~\ref{S-mixing ultrabarrelled} we characterize mixing unilateral backward shifts on ultrabarrelled topological sequence spaces. In addition, we exhibit one mixing unilateral backward shift on an ultrabarrelled but not barrelled topological sequence space in Section 3. This example shows that our results are effective.

\noindent{\it Acknowledgments.}
Z.~R. was supported by National Natural Science Foundation of China (Grant No.12261063).

\section{In the case of barrelled spaces} \label{S-mixing barrelled}
In this section we will characterize mixing unilateral backward shifts on barrelled topological sequence spaces.

Recall the following notions introduced in \cite{Wilansky78}.

\begin{definition}
A subset $A$ of a topological vector space $X$ is called {\it bounded} if for every neighborhood $U$ of 0 there exists $\varepsilon>0$ such that $tA\subseteq U$ whenever $|t|<\varepsilon$.

It is known that every convergent sequence in a topological vector space is bounded (see \cite[page 49, Problem 4.4.10]{Wilansky78}).

Let $X,Y$ be two topological vector spaces, $\mathfrak{F}$ a set of linear maps from $X$ to $Y$. Then $\mathfrak{F}$
is called {\it equicontinuous} if for each neighborhood $V$ of 0 in $Y$ there exists a neighborhood $U$ of 0 in $X$ such that $f(U)\subseteq V$
for each $f\in\mathfrak{F}$.
\end{definition}

We need the following Banach-Steinhaus Theorem, see \cite[page 137, Theorem 9.3.4]{Wilansky78}.

\begin{theorem}[\cite{Wilansky78}]
If $\mathfrak{F}$ is a collection of continuous linear maps from a barrelled space $X$ into a locally convex space $Y$, and if the set
$$\mathfrak{F}(\mu)=\{f(\mu):f\in\mathfrak{F}\}$$
is bounded in $Y$ for every $\mu\in X$, then $\mathfrak{F}$ is equicontinuous.
\end{theorem}


We start by studying the unweighted unilateral backward shifts on barrelled topological sequence spaces. Our results will then extend immediately to all weighted unilateral backward shifts on barrelled topological sequence spaces via a simple conjugacy.

The following theorem is an important criterion for mixing operators in linear dynamical systems which is called Kitai's Criterion, see \cite[page 345, Theorem 12.31]{Grosse-Erdmann-Peris} and \cite{Kitai}.

\begin{theorem}[\cite{Grosse-Erdmann-Peris,Kitai}]
Let $S$ be a continuous linear operator on a topological vector space $X$. If there are dense subsets $X_{0},Y_{0}\subseteq X$ and a map $F:Y_{0}\rightarrow Y_{0}$
such that, for any $\mu\in X_{0},\nu\in Y_{0}$,
\begin{enumerate}
  \item $\lim\limits_{n\rightarrow\infty}S^{n}\mu=0$,
  \item $\lim\limits_{n\rightarrow\infty}F^{n}\nu=0$,
  \item $SF\nu=\nu$,
\end{enumerate}
then $S$ is mixing.
\end{theorem}


\begin{proposition}
Let $X$ be a topological sequence space in which $(e_{n})_{n=1}^{\infty}$ form a basis. Suppose that the unweighted unilateral backward shift $T$ is a continuous linear operator on $X$, where
$$T(u_{1},u_{2},u_{3},\cdots)=(u_{2},u_{3},u_{4},\cdots).$$
Suppose that $\lim\limits_{n\rightarrow\infty}e_{n}=0$. Then $T$ is mixing.
\end{proposition}

\begin{proof}
We apply Theorem 2.3. Since $(e_{n})_{n=1}^{\infty}$ is a basis and the convergence in $X$ imply coordinatewise convergence, for any $\mu=(u_{n})_{n=1}^{\infty}\in X$ we have $\mu=\sum\limits_{n=1}^{\infty}u_{n}e_{n}$. Hence $span\{e_{n}:n\geqslant1\}$ is dense in $X$, where $span\{e_{n}:n\geqslant1\}$ is the collection of all linear combinations of finite elements of $\{e_{n}:n\geqslant1\}$. Take $X_{0}=Y_{0}=span\{e_{n}:n\geqslant1\}$. For $F:Y_{0}\rightarrow Y_{0}$ we take the unweighted unilateral forward shift $F(u_{1},u_{2},u_{3},\cdots)=(0,u_{1},u_{2},\cdots)$.

(1) For any $\mu\in X_{0}$, when $n$ is sufficiently large we have $T^{n}\mu=0$, then $T^{n}\mu\rightarrow0$ in $X$ as $n\rightarrow\infty$.

(2) For any $\nu\in Y_{0}$, we will show that $\lim\limits_{n\rightarrow\infty}F^{n}\nu=0$.

First we will show that $\lim\limits_{n\rightarrow\infty}F^{n}e_{j}=0$ for each $j\geqslant1$. For any $j\geqslant1$, since $F^{n}e_{j}=e_{j+n}$ and $\lim\limits_{n\rightarrow\infty}e_{n}=0$, $\lim\limits_{n\rightarrow\infty}F^{n}e_{j}=0$.

Next we will show that $\lim\limits_{n\rightarrow\infty}F^{n}\nu=0$ for each $\nu\in Y_{0}$. For any $\nu\in Y_{0}$, there exist a positive integer $m$ and scalars $\alpha_{1},\cdots,\alpha_{m}\in\mathbb{K}$ such that $\nu=\alpha_{1}e_{1}+\cdots+\alpha_{m}e_{m}$. Since for each $1\leqslant j\leqslant m$ we have $\lim\limits_{n\rightarrow\infty}F^{n}e_{j}=0$, $F^{n}\nu=\alpha_{1}F^{n}e_{1}+\cdots+\alpha_{m}F^{n}e_{m}\rightarrow0$ in $X$ as $n\rightarrow\infty$.

(3) For any $\nu\in Y_{0}$, we have $TF\nu=\nu$.

This shows that conditions (1)--(3) of Theorem 2.3 hold, so that $T$ is mixing.
\end{proof}

\begin{example}
The unweighted unilateral backward shift
$$T:(c_{0},\sigma(c_{0},l^{1}))\rightarrow(c_{0},\sigma(c_{0},l^{1}))$$
is mixing. Here $\sigma(c_{0},l^{1})$ is the weak topology of $c_{0}$. For the more details about the weak topology please refer to \cite[page 213]{Megginson}.

It is obvious that $(c_{0},\sigma(c_{0},l^{1}))$ is a topological sequence space in which $(e_{n})_{n=1}^{\infty}$ form a basis. Since $(c_{0},\|\cdot\|_{\infty})$ is an infinite dimensional Banach space, the weak topology of $c_{0}$ can not be metrizable (see \cite[page 215, Proposition 2.5.14]{Megginson}). Hence $(c_{0},\sigma(c_{0},l^{1}))$ is not a Fr$\acute{e}$chet sequence space. Clearly the unweighted unilateral backward shift $$T:(c_{0},\sigma(c_{0},l^{1}))\rightarrow(c_{0},\sigma(c_{0},l^{1}))$$
is continuous. Finally we claim that $e_{n}\rightarrow0$ in $(c_{0},\sigma(c_{0},l^{1}))$ as $n\rightarrow\infty$. For each $(\beta_{k})_{k=1}^{\infty}\in l^{1}$, we have $\lim\limits_{n\rightarrow\infty}\beta_{n}=0$. This implies that $e_{n}\rightarrow0$ in $(c_{0},\sigma(c_{0},l^{1}))$ as $n\rightarrow\infty$. By Proposition 2.4 we have the unweighted unilateral backward shift
$$T:(c_{0},\sigma(c_{0},l^{1}))\rightarrow(c_{0},\sigma(c_{0},l^{1}))$$
is mixing.

\end{example}

\begin{flushright}
  $\Box$
\end{flushright}

\begin{example}
The unweighted unilateral backward shift
$$T:(l^{\infty},\sigma(l^{\infty},l^{1}))\rightarrow(l^{\infty},\sigma(l^{\infty},l^{1}))$$
is mixing. Here $\sigma(l^{\infty},l^{1})$ is the weak$^{\ast}$ topology of $l^{\infty}$. For the more details about the weak$^{\ast}$ topology please refer to \cite[page 224]{Megginson}.

It is obvious that $(l^{\infty},\sigma(l^{\infty},l^{1}))$ is a topological sequence space in which $(e_{n})_{n=1}^{\infty}$ form a basis.
Since $(l^{1},\|\cdot\|_{1})$ is an infinite dimensional Banach space, the weak$^{\ast}$ topology of $(l^{1})^{\ast}$ can not be metrizable (see \cite[page 226, Proposition 2.6.12]{Megginson}). Hence $\sigma(l^{\infty},l^{1})$ can not be metrizable. Therefore $(l^{\infty},\sigma(l^{\infty},l^{1}))$ is not a Fr$\acute{e}$chet sequence space. Clearly the unweighted unilateral backward shift
$$T:(l^{\infty},\sigma(l^{\infty},l^{1}))\rightarrow(l^{\infty},\sigma(l^{\infty},l^{1}))$$
is continuous. Finally we claim that $e_{n}\rightarrow0$ in $(l^{\infty},\sigma(l^{\infty},l^{1}))$ as $n\rightarrow\infty$. For each $(\beta_{k})_{k=1}^{\infty}\in l^{1}$, we have $\lim\limits_{n\rightarrow\infty}\beta_{n}=0$. This implies that $e_{n}\rightarrow0$ in $(l^{\infty},\sigma(l^{\infty},l^{1}))$ as $n\rightarrow\infty$. By Proposition 2.4 we have the unweighted unilateral backward shift
$$T:(l^{\infty},\sigma(l^{\infty},l^{1}))\rightarrow(l^{\infty},\sigma(l^{\infty},l^{1}))$$
is mixing.
\end{example}

\begin{flushright}
  $\Box$
\end{flushright}

\begin{proposition}
Let $X$ be a barrelled topological sequence space in which $(e_{n})_{n=1}^{\infty}$ form a basis. Suppose that the unweighted unilateral backward shift $T$ is a mixing linear operator on $X$, where
$$T(u_{1},u_{2},u_{3},\cdots)=(u_{2},u_{3},u_{4},\cdots).$$
Then $\lim\limits_{n\rightarrow\infty}e_{n}=0$.
\end{proposition}

\begin{proof}
It suffices to show that for any neighborhood $U$ of $0$ in $X$, there is a positive integer $N$ such that $e_{n}\in U$ whenever $n>N$.

Since $(e_{k})_{k=1}^{\infty}$ is a basis and the convergence in $X$ imply coordinatewise convergence, for any $\mu=(u_{k})_{k=1}^{\infty}\in X$, the representation $\mu=\sum\limits_{k=1}^{\infty}u_{k}e_{k}$ converges in $X$. Hence for any $\mu=(u_{k})_{k=1}^{\infty}\in X$, the sequence $(u_{k}e_{k})_{k=1}^{\infty}$ converges to 0 in $X$.

Since $X$ is a topological vector space, the addition operation is continuous. Hence there is a neighborhood $V_{1}$ of 0 such that $V_{1}+V_{1}\subseteq U$, where $V_{1}+V_{1}=\{x+y:x,y\in V_{1}\}$. Since every neighborhood of 0 in a topological vector space includes a balanced neighborhood of 0 (see \cite[page 41, Lemma 4.2.5]{Wilansky78}), we may choose a balanced neighborhood $V$ of 0 such that $V\subseteq V_{1}$. Then $V+V\subseteq U$.

For each positive integer $n$, we define a continuous linear operator $f_{n}$ on $X$ by $f_{n}(\mu)=u_{n}e_{n}$, where $\mu=(u_{k})_{k=1}^{\infty}\in X$. Since the sequence $(u_{k}e_{k})_{k=1}^{\infty}$ converges to 0 in $X$, for each $\mu\in X$, $\{f_{n}(\mu):n\geqslant1\}$ is a bounded set in $X$. By Theorem 2.2, applied to the continuous linear operators $f_{n},n\geqslant1$, there is some open neighborhood $W$ of 0 such that, for all $\mu=(u_{k})_{k=1}^{\infty}\in X$,
\begin{align*}
\mu\in W\Rightarrow u_{n}e_{n}\in V
\end{align*}
for all $n\geqslant1$.

Moreover, since the convergence in $X$ imply coordinatewise convergence, there is some open neighborhood $Q$ of 0 such that, for all $\mu=(u_{k})_{k=1}^{\infty}\in X$,
\begin{align*}
\mu\in Q\Rightarrow|u_{1}|\leqslant\frac{1}{2}.
\end{align*}

Since $T$ is mixing, there is some positive integer $N$ such that $T^{n-1}(W)\bigcap(e_{1}+Q)\neq\emptyset$ whenever $n>N$.

We will show that $e_{n}\in U$ whenever $n>N$.

For each $n>N$, we may choose $\mu_{n}\in W$ with $T^{n-1}(\mu_{n})\in e_{1}+Q$. Take $\mu_{n}=(u_{n,k})_{k=1}^{\infty}$. Then $T^{n-1}(\mu_{n})=(u_{n,n},u_{n,n+1},\cdots)$.
Since $\mu_{n}\in W$ and $T^{n-1}(\mu_{n})-e_{1}\in Q$, we have $u_{n,n}e_{n}\in V$ and $|u_{n,n}-1|\leqslant\frac{1}{2}$. Since $|u_{n,n}-1|\leqslant\frac{1}{2}$, we have $|u_{n,n}|\geqslant\frac{1}{2}$ and
$$|u_{n,n}^{-1}-1|=|\frac{1-u_{n,n}}{u_{n,n}}|\leqslant1.$$
Hence
$$e_{n}=(u_{n,n}^{-1}-1)u_{n,n}e_{n}+u_{n,n}e_{n}\in(u_{n,n}^{-1}-1)V+V.$$
Since $V$ is balanced, $(u_{n,n}^{-1}-1)V\subseteq V$. Therefore $e_{n}\in V+V$. Notice that $V+V\subseteq U$, we have $e_{n}\in U$ whenever $n>N$.
\end{proof}

Using Propositions 2.4 and 2.7 we immediately get Theorem 1.3.

Grosse-Erdmann and Peris \cite[Theorem 4.5]{Grosse-Erdmann-Peris} proved Theorem 1.3 in the case of Fr$\acute{e}$chet sequence spaces, thus Theorem 1.3 generalizes Theorem 1.1.

\begin{example}
The space $\varphi$ is defined to consist of all sequences $\mu=(u_{k})_{k=1}^{\infty}$ of scalars such that $\{k\in\mathbb{N}^{\ast}:u_{k}\neq0\}$ is a finite subset of $\mathbb{N}^{\ast}$. Let $\mathfrak{B}$ be the collection of absolutely convex absorbing sets in $\varphi$. Then there is a unique locally convex topology for $\varphi$ having $\mathfrak{B}$ as a neighborhood base at 0. We denote this unique locally convex topology by $\mathfrak{T}$ and call it the {\it largest locally convex topology} for $\varphi$. For the more details about the largest locally convex topology please refer to \cite[page 93, Problems 7.1.5, 7.1.7, 7.1.8]{Wilansky78}.

In this example we will show that the unweighted unilateral backward shift
$$T:(\varphi, \mathfrak{T})\rightarrow(\varphi, \mathfrak{T})$$
is not mixing.

It is obvious that $(\varphi, \mathfrak{T})$ is a barrelled topological sequence space in which $(e_{n})_{n=1}^{\infty}$ form a basis. Since $\varphi$ is an infinite-dimensional vector space, $(\varphi, \mathfrak{T})$ cannot be metrizable (see \cite[page 93, Problem 7.1.8]{Wilansky78}). Hence $(\varphi, \mathfrak{T})$ is not a Fr$\acute{e}$chet sequence space. Since every linear map from $(\varphi, \mathfrak{T})$ to a locally convex space is continuous (see \cite[page 93, Problem 7.1.7]{Wilansky78}), the unweighted unilateral backward shift
$$T:(\varphi, \mathfrak{T})\rightarrow(\varphi, \mathfrak{T})$$
is continuous. Finally we claim that the sequence $\{e_{n}\}_{n=1}^{\infty}$ does not converge to 0 in $(\varphi, \mathfrak{T})$. Suppose that $e_{n}\rightarrow0$ in $(\varphi, \mathfrak{T})$ as $n\rightarrow\infty$. Then for any locally convex topology $\tau$ for $\varphi$, we have $e_{n}\rightarrow0$ in $(\varphi, \tau)$ as $n\rightarrow\infty$. In particular, $e_{n}\rightarrow0$ in $(\varphi, \|\cdot\|)$ as $n\rightarrow\infty$, where the norm $\|\cdot\|$ in $\varphi$ is defined by
$$\|\mu\|=\sup\limits_{k\geqslant1}|u_{k}| (\forall\mu=(u_{k})_{k=1}^{\infty}\in\varphi).$$
While in fact $\|e_{n}\|=1$, $\{e_{n}\}_{n=1}^{\infty}$ does not converge to 0 in $(\varphi, \|\cdot\|)$. This is a contradiction. Thus $\{e_{n}\}_{n=1}^{\infty}$ does not converge to 0 in $(\varphi, \mathfrak{T})$. By Theorem 1.3 we have the unweighted unilateral backward shift
$$T:(\varphi, \mathfrak{T})\rightarrow(\varphi, \mathfrak{T})$$
is not mixing.
\end{example}

\begin{flushright}
  $\Box$
\end{flushright}

It is now an easy matter to transfer our results so far to arbitrary weighted unilateral backward shifts by means of a suitable conjugacy.

Let $\nu=(v_{n})_{n=1}^{\infty}$ be a sequence of nonzero scalars. Let $(X,\mathfrak{T})$ be a topological sequence space and $T_{\nu}$ be the unilateral weighted backward shift on $(X,\mathfrak{T})$, where
$$T_{\nu}(u_{1},u_{2},u_{3},\cdots)=(v_{2}u_{2},v_{3}u_{3},v_{4}u_{4},\cdots).$$
We define new weights $w_{n}$ by
$$w_{n}=(\prod_{i=1}^{n}v_{i})^{-1},n\geqslant1,$$
and consider the sequence space
$$X_{\omega}=\{(u_{n})_{n=1}^{\infty}\in\mathbb{K}^{\mathbb{N}^{\ast}}:(u_{n}w_{n})_{n=1}^{\infty}\in X\}.$$

Clearly the unweighted unilateral backward shift $T$ is a linear operator on $X_{\omega}$. In fact, if $(u_{n})_{n=1}^{\infty}\in X_{\omega}$, then $(u_{n}w_{n})_{n=1}^{\infty}\in X$. Since $T_{\nu}$ is a linear operator on $X$, $(v_{n+1}u_{n+1}w_{n+1})_{n=1}^{\infty}\in X$. Notice that $v_{n+1}u_{n+1}w_{n+1}=u_{n+1}w_{n}$, then $(u_{n+1}w_{n})_{n=1}^{\infty}\in X$. Therefore $(u_{n+1})_{n=1}^{\infty}\in X_{\omega}$. This implies that the unweighted unilateral backward shift $T$ is a linear operator on $X_{\omega}$.

The map $\phi_{\omega}:X_{\omega}\rightarrow X,(u_{n})_{n=1}^{\infty}\mapsto(u_{n}w_{n})_{n=1}^{\infty}$ is a vector space isomorphism. We may use $\phi_{\omega}$ to transfer a topology from $X$ to $X_{\omega}$. Take $\mathfrak{T}_{\omega}=\{U\subseteq X_{\omega}:\phi_{\omega}(U)\in\mathfrak{T}\}$. Clearly $\mathfrak{T}_{\omega}$ is a topology on $X_{\omega}$. Since $(X,\mathfrak{T})$ is a topological sequence space, $(X_{\omega},\mathfrak{T}_{\omega})$ is also a topological sequence space. If $(X,\mathfrak{T})$ is barrelled, then $(X_{\omega},\mathfrak{T}_{\omega})$ is also barrelled. If $(e_{n})_{n=1}^{\infty}$ form a basis for $(X,\mathfrak{T})$, then $(e_{n})_{n=1}^{\infty}$ also form a basis for $(X_{\omega},\mathfrak{T}_{\omega})$.

Finally, a simple calculation shows that $T_{\nu}\circ\phi_{\omega}=\phi_{\omega}\circ T$. Hence $T_{\nu}:(X,\mathfrak{T})\rightarrow(X,\mathfrak{T})$ is continuous if and only if $T:(X_{\omega},\mathfrak{T}_{\omega})\rightarrow(X_{\omega},\mathfrak{T}_{\omega})$ is continuous, and $T_{\nu}:(X,\mathfrak{T})\rightarrow(X,\mathfrak{T})$ is mixing if and only if $T:(X_{\omega},\mathfrak{T}_{\omega})\rightarrow(X_{\omega},\mathfrak{T}_{\omega})$ is mixing. Therefore our previous results immediately yield Theorem 1.4.

$\mathbf{Proof~of~Theorem~1.4.}$

Since $X$ is a barrelled topological sequence space in which $(e_{n})_{n=1}^{\infty}$ form a basis, $X_{\omega}$ is a barrelled topological sequence space in which $(e_{n})_{n=1}^{\infty}$ form a basis. Since the unilateral weighted backward shift $T_{\nu}$ is a continuous linear operator on $X$, the unweighted unilateral backward shift $T$ is a continuous linear operator on $X_{\omega}$.

\begin{align*}
&T_{\nu}:X\rightarrow X\text{ is mixing }\\
&\Longleftrightarrow T:X_{\omega}\rightarrow X_{\omega}\text{ is mixing }\\
&\Longleftrightarrow e_{n}\rightarrow0\text{ in }X_{\omega}\text{ as } n\rightarrow\infty\text{ (By Theorem 1.3)}\\
&\Longleftrightarrow \phi_{\omega}(e_{n})\rightarrow0\text{ in }X\text{ as } n\rightarrow\infty\\
&\Longleftrightarrow (\prod\limits_{i=1}^{n}v_{i})^{-1}e_{n}\rightarrow0\text{ in }X\text{ as } n\rightarrow\infty.
\end{align*}

\begin{flushright}
  $\Box$
\end{flushright}

Grosse-Erdmann and Peris \cite[Theorem 4.8 (b)]{Grosse-Erdmann-Peris} proved Theorem 1.4 in the case of Fr$\acute{e}$chet sequence spaces, thus Theorem 1.4 generalizes Theorem 1.2.

\begin{example}
Let $Y=\bigcup\limits_{n=1}^{\infty}l^{n}$. And
$$\mathfrak{B}=\{V\subseteq Y:V\text{ is absolutely convex and }V\bigcap l^{n}\text{ is a neighborhood of 0 in }(l^{n},\|\cdot\|_{n})$$
$\text{ for each }n\geqslant1\}.$

Then there is a unique locally convex topology for $Y$ having $\mathfrak{B}$ as a neighborhood base at 0. We denote this unique locally convex topology by $\mathfrak{T}$ and call it the {\it inductive limit topology} for $Y$. For the more details about the inductive limit topology please refer to \cite[pages 209-212]{Wilansky78}.

In this example we will show that the unilateral weighted backward shift
$$T_{\nu}:(Y,\mathfrak{T})\rightarrow(Y,\mathfrak{T})$$
is mixing, where $\nu=(v_{k})_{k=1}^{\infty}$ is a bounded sequence of nonzero scalars with
$$\lim\limits_{k\rightarrow\infty}\prod\limits_{i=1}^{k}|v_{i}|=+\infty.$$

We use Theorem 1.4 to show this example.

(1) The space $(Y,\mathfrak{T})$ is a topological sequence space. For each positive integer $i$, we will prove that the coordinate functional $P_{i}:(Y,\mathfrak{T})\rightarrow\mathbb{K}$ defined by $P_{i}\mu=u_{i} (\forall\mu=(u_{k})_{k=1}^{\infty}\in Y)$ is continuous. Since the coordinate functional $P_{i}:(Y,\mathfrak{T})\rightarrow\mathbb{K}$ is continuous if and only if $P_{i}|_{l^{n}}:(l^{n},\|\cdot\|_{n})\rightarrow\mathbb{K}$ is continuous for each $n\geqslant1$ (see \cite[page 210, Theorem 13.1.8]{Wilansky78}), we only have to show that $P_{i}|_{l^{n}}:(l^{n},\|\cdot\|_{n})\rightarrow\mathbb{K}$ is continuous for any $n\geqslant1$. While this is obvious. Hence the convergence in $(Y,\mathfrak{T})$ implies coordinatewise convergence. In addition, $(Y,\mathfrak{T})$ is a locally convex space and $Y$ is a linear subspace of $\mathbb{K}^{\mathbb{N}^{\ast}}$. So $(Y,\mathfrak{T})$ is a topological sequence space.

(2) The space $(Y,\mathfrak{T})$ is not a Fr$\acute{e}$chet sequence space. Since $(Y,\mathfrak{T})$ can not be metrizable (see \cite[page 65, Example 4.1.2]{Wilansky67}), $(Y,\mathfrak{T})$ is not a Fr$\acute{e}$chet sequence space.

(3) The space $(Y,\mathfrak{T})$ is barrelled. Let $B$ be a barrel in $(Y,\mathfrak{T})$. For each $n\in\mathbb{N}^{\ast}$, $B\bigcap l^{n}$ is a barrel in $(l^{n},\|\cdot\|_{n})$. Since each $(l^{n},\|\cdot\|_{n})$ is barrelled, $B\bigcap l^{n}$ is a neighborhood of 0 in $(l^{n},\|\cdot\|_{n})$. Therefore $B$ is a neighborhood of 0 in $(Y,\mathfrak{T})$. Thus $(Y,\mathfrak{T})$ is barrelled.

(4) $(e_{k})_{k=1}^{\infty}$ form a basis for $(Y,\mathfrak{T})$. For each $\mu=(u_{k})_{k=1}^{\infty}\in Y$, we will show that the sequence $(\sum\limits_{i=1}^{k}u_{i}e_{i})_{k=1}^{\infty}$ converges to $\mu$ in $(Y,\mathfrak{T})$. Let $V$ be an absolutely convex subset of $Y$ and $V\bigcap l^{n}$ is a neighborhood of 0 in $(l^{n},\|\cdot\|_{n})$ for each $n\geqslant1$. We will show that there is some positive integer $k_{0}$ such that $\sum\limits_{i=1}^{k}u_{i}e_{i}-\mu\in V$ whenever $k>k_{0}$. Since $\mu\in Y$, there is some positive integer $m$ such that $\mu\in l^{m}$. Notice that $\sum\limits_{i=1}^{k}u_{i}e_{i}\in l^{m}$ and the sequence $(\sum\limits_{i=1}^{k}u_{i}e_{i})_{k=1}^{\infty}$ converges to $\mu$ in $(l^{m},\|\cdot\|_{m})$. Since $V\bigcap l^{m}$ is a neighborhood of 0 in $(l^{m},\|\cdot\|_{m})$, there is some positive integer $k_{0}$ such that $\sum\limits_{i=1}^{k}u_{i}e_{i}-\mu\in V\bigcap l^{m}\subseteq V$ whenever $k>k_{0}$. Thus the sequence $(\sum\limits_{i=1}^{k}u_{i}e_{i})_{k=1}^{\infty}$ converges to $\mu$ in $(Y,\mathfrak{T})$. Since convergence in $(Y,\mathfrak{T})$ imply coordinatewise convergence, every $\mu=(u_{k})_{k=1}^{\infty}\in Y$ has a unique representation $\mu=\sum\limits_{k=1}^{\infty}u_{k}e_{k}$. Hence $(e_{k})_{k=1}^{\infty}$ form a basis for $(Y,\mathfrak{T})$.

(5) The unilateral weighted backward shift
$$T_{\nu}:(Y,\mathfrak{T})\rightarrow(Y,\mathfrak{T})$$
is continuous. Let $n$ be an any positive integer. Clearly $T_{\nu}|_{l^{n}}=i\circ P$, where $P:(l^{n},\|\cdot\|_{n})\rightarrow(l^{n},\|\cdot\|_{n})$ is defined by
$$P(u_{1},u_{2},u_{3},\cdots)=(v_{2}u_{2},v_{3}u_{3},v_{4}u_{4},\cdots),$$
and $i:(l^{n},\|\cdot\|_{n})\rightarrow(Y,\mathfrak{T})$ is the inclusion map defined by
$$i(u_{1},u_{2},u_{3},\cdots)=(u_{1},u_{2},u_{3},\cdots).$$
Since $\sup\limits_{k\geqslant1}|v_{k}|<+\infty$, $P:(l^{n},\|\cdot\|_{n})\rightarrow(l^{n},\|\cdot\|_{n})$ is continuous. Since
$$\mathfrak{B}=\{V\subseteq Y:V\text{ is absolutely convex and }V\bigcap l^{n}\text{ is a neighborhood of 0 in }(l^{n},\|\cdot\|_{n})$$
$\text{ for each }n\geqslant1\}$
is a neighborhood base at 0 for $(Y,\mathfrak{T})$, the inclusion map $i:(l^{n},\|\cdot\|_{n})\rightarrow(Y,\mathfrak{T})$ is continuous. Thus $T_{\nu}|_{l^{n}}:(l^{n},\|\cdot\|_{n})\rightarrow(Y,\mathfrak{T})$ is continuous. By \cite[page 210, Theorem 13.1.8]{Wilansky78},
$$T_{\nu}:(Y,\mathfrak{T})\rightarrow(Y,\mathfrak{T})$$
is continuous.

(6) $(\prod\limits_{i=1}^{k}v_{i})^{-1}e_{k}\rightarrow0$ in $(Y,\mathfrak{T})$ as $k\rightarrow\infty$. Let $V$ be an absolutely convex subset of $Y$ and $V\bigcap l^{n}$ is a neighborhood of 0 in $(l^{n},\|\cdot\|_{n})$ for each $n\geqslant1$. We will show that there is some positive integer $k_{0}$ such that $(\prod\limits_{i=1}^{k}v_{i})^{-1}e_{k}\in V$ whenever $k>k_{0}$. Since $\lim\limits_{k\rightarrow\infty}\prod\limits_{i=1}^{k}|v_{i}|=+\infty$, the sequence $((\prod\limits_{i=1}^{k}v_{i})^{-1}e_{k})_{k=1}^{\infty}$ converges to 0 in $(l^{1},\|\cdot\|_{1})$, and hence there is some positive integer $k_{0}$ such that $(\prod\limits_{i=1}^{k}v_{i})^{-1}e_{k}\in V\bigcap l^{1}\subseteq V$ whenever $k>k_{0}$. Thus $(\prod\limits_{i=1}^{k}v_{i})^{-1}e_{k}\rightarrow0$ in $(Y,\mathfrak{T})$ as $k\rightarrow\infty$.

By Theorem 1.4 we have the unilateral weighted backward shift
$$T_{\nu}:(Y,\mathfrak{T})\rightarrow(Y,\mathfrak{T})$$
is mixing.
\end{example}

\begin{flushright}
  $\Box$
\end{flushright}

\section{In the case of ultrabarrelled spaces} \label{S-mixing ultrabarrelled}
In this section we will characterize mixing unilateral backward shifts on ultrabarrelled topological sequence spaces.

First we will state and prove an analogue of the Banach-Steinhaus Theorem 2.2 for $X$ ultrabarrelled and $Y$ a topological vector space.

We need the following Theorem which gives the most general construction of a vector topology, see \cite[page 45, Theorem 4.3.5]{Wilansky78}.

\begin{theorem}[\cite{Wilansky78}]
Let $X$ be a vector space. Let $\mathfrak{B}$ be a nonempty collection of nonempty subsets of $X$ such that $\mathfrak{B}$ satisfies the following properties:

(1) Each member of $\mathfrak{B}$ is balanced and absorbing.

(2) If $U\in\mathfrak{B}$, there exists $V\in\mathfrak{B}$ such that $V+V\subseteq U$.

(3) If $U_{1}$ and $U_{2}$ are in $\mathfrak{B}$, there exists $U_{3}\in\mathfrak{B}$ such that $U_{3}\subseteq U_{1}\bigcap U_{2}$.

Then there is a unique vector topology for $X$ for which $\mathfrak{B}$ is a neighborhood base at 0.
\end{theorem}

\begin{theorem}
If $\mathfrak{F}$ is a collection of continuous linear maps from an ultrabarrelled $(X,\mathfrak{T})$ into a topological vector space $Y$, and if the set
$$\mathfrak{F}(\mu)=\{f(\mu):f\in\mathfrak{F}\}$$
is bounded in $Y$ for every $\mu\in X$, then $\mathfrak{F}$ is equicontinuous.
\end{theorem}

\begin{proof}
Let $V$ be a neighborhood of 0 in $Y$. We will show that there exists a neighborhood $U$ of 0 in $(X,\mathfrak{T})$ such that $f(U)\subseteq V$
for each $f\in\mathfrak{F}$. Since every neighborhood of 0 in a topological vector space includes a closed balanced neighborhood of 0 (see \cite[page 41, Theorem 4.2.7]{Wilansky78}), we may choose a closed balanced neighborhood $W$ of 0 in $Y$ such that $W\subseteq V$. Next we will show that there exists a neighborhood $U$ of 0 in $(X,\mathfrak{T})$ such that $f(U)\subseteq W$ for each $f\in\mathfrak{F}$. Let $U=\bigcap\limits_{f\in\mathfrak{F}}f^{-1}(W)$. Then $f(U)\subseteq W$ for each $f\in\mathfrak{F}$. We claim that $U$ is a neighborhood of 0 in $(X,\mathfrak{T})$. Let $$\mathfrak{B}=\{\bigcap\limits_{f\in\mathfrak{F}}f^{-1}(P):P\text{ is a closed balanced neighborhood of 0 in }Y\}.$$
Then $\mathfrak{B}$ satisfies the following properties:

(1) Each member of $\mathfrak{B}$ is closed in $(X,\mathfrak{T})$.

Let $P$ be a closed balanced neighborhood of 0 in $Y$. Since each $f\in\mathfrak{F}$ is continuous, $f^{-1}(P)$ is closed in $(X,\mathfrak{T})$. Thus $\bigcap\limits_{f\in \mathfrak{F}}f^{-1}(P)$ is closed in $(X,\mathfrak{T})$.

(2) Each member of $\mathfrak{B}$ is balanced and absorbing.

Let $P$ be a closed balanced neighborhood of 0 in $Y$. Since $P$ is balanced, $\bigcap\limits_{f\in\mathfrak{F}}f^{-1}(P)$ is balanced. Next we will show that $\bigcap\limits_{f\in \mathfrak{F}}f^{-1}(P)$ is absorbing. For every $\mu\in X$, since the set $\mathfrak{F}(\mu)=\{f(\mu):f\in\mathfrak{F}\}$ is bounded in $Y$, there exists $\varepsilon>0$ such that $t\mathfrak{F}(\mu)\subseteq P$ whenever $|t|<\varepsilon$. Hence $t\mu\in\bigcap\limits_{f\in\mathfrak{F}}f^{-1}(P)$ for $|t|<\varepsilon$, this implies that $\bigcap\limits_{f\in \mathfrak{F}}f^{-1}(P)$ is absorbing.

(3) Let $P$ be a closed balanced neighborhood of 0 in $Y$. We will show that there exists a closed balanced neighborhood $Q$ of 0 in $Y$ such that $\bigcap\limits_{f\in \mathfrak{F}}f^{-1}(Q)+\bigcap\limits_{f\in\mathfrak{F}}f^{-1}(Q)\subseteq\bigcap\limits_{f\in\mathfrak{F}}f^{-1}(P)$. Since $Y$ is a topological vector space, addition of vectors is a continuous operation from $Y\times Y$ into $Y$. Hence there is a neighborhood $P_{1}$ of 0 in $Y$ such that $P_{1}+P_{1}\subseteq P$. Since every neighborhood of 0 in a topological vector space includes a closed balanced neighborhood of 0 (see \cite[page 41, Theorem 4.2.7]{Wilansky78}), we may choose a closed balanced neighborhood $Q$ of 0 in $Y$ such that $Q\subseteq P_{1}$. Thus $Q+Q\subseteq P$ and $\bigcap\limits_{f\in\mathfrak{F}}f^{-1}(Q)+\bigcap\limits_{f\in\mathfrak{F}}f^{-1}(Q)\subseteq\bigcap\limits_{f\in\mathfrak{F}}f^{-1}(P)$.

(4) Let $P,Q$ be two closed balanced neighborhoods of 0 in $Y$. Then $P\bigcap Q$ is also a closed balanced neighborhood of 0 in $Y$ such that $(\bigcap\limits_{f\in \mathfrak{F}}f^{-1}(P))\bigcap(\bigcap\limits_{f\in\mathfrak{F}}f^{-1}(Q))=\bigcap\limits_{f\in\mathfrak{F}}f^{-1}(P\bigcap Q)\in\mathfrak{B}$.

By Theorem 3.1, there is a unique vector topology $\mathfrak{T}^{'}$ for $X$ for which $\mathfrak{B}$ is a neighborhood base at 0. Clearly $\mathfrak{T}^{'}$ is $F$ linked to $\mathfrak{T}$. Since $(X,\mathfrak{T})$ is ultrabarrelled, we have $\mathfrak{T}^{'}\subseteq\mathfrak{T}$. Notice that $U=\bigcap\limits_{f\in \mathfrak{F}}f^{-1}(W)\in\mathfrak{B}$. Hence
$U$ is a neighborhood of 0 in $(X,\mathfrak{T})$. This proves our claim.
\end{proof}

Next we start by studying the unweighted unilateral backward shifts on ultrabarrelled topological sequence spaces. Our results will then extend immediately to all weighted unilateral backward shifts on ultrabarrelled topological sequence spaces via a simple conjugacy.

\begin{proposition}
Let $X$ be an ultrabarrelled topological sequence space in which $(e_{n})_{n=1}^{\infty}$ form a basis. Suppose that the unweighted unilateral backward shift $T$ is a mixing linear operator on $X$, where
$$T(u_{1},u_{2},u_{3},\cdots)=(u_{2},u_{3},u_{4},\cdots).$$
Then $\lim\limits_{n\rightarrow\infty}e_{n}=0$.
\end{proposition}

The proof of Proposition 3.3 is identical to that in Proposition 2.7 if one applies Theorem 3.2 to the continuous linear operators $f_{n},n\geqslant1$, where $f_{n}(\mu)=u_{n}e_{n} (\forall\mu=(u_{k})_{k=1}^{\infty}\in X)$. Hence we omit its proof.

Using Propositions 2.4 and 3.3 we immediately get Theorem 1.5.

Grosse-Erdmann and Peris \cite[Theorem 4.5]{Grosse-Erdmann-Peris} proved Theorem 1.5 in the case of Fr$\acute{e}$chet sequence spaces, thus Theorem 1.5 generalizes Theorem 1.1.

Using the conjugacy in Theorem 1.4 we get Theorem 1.6.

The proof of Theorem 1.6 is identical to that in Theorem 1.4 if one uses the conjugacy in Theorem 1.4 and Theorem 1.5, and hence we omit its proof.

Grosse-Erdmann and Peris \cite[Theorem 4.8 (b)]{Grosse-Erdmann-Peris} proved Theorem 1.6 in the case of Fr$\acute{e}$chet sequence spaces, thus Theorem 1.6 generalizes Theorem 1.2.

\begin{example}
Let $0<p<1$. The space $l^{p}$ is defined to consist of all sequences $\mu=(u_{k})_{k=1}^{\infty}$ of scalars such that $\sum\limits_{k=1}^{\infty}|u_{k}|^{p}<+\infty$. The metric in $l^{p}$ is defined by
$$d(\mu,\nu)=\sum\limits_{k=1}^{\infty}|u_{k}-v_{k}|^{p} (\forall\mu=(u_{k})_{k=1}^{\infty},\nu=(v_{k})_{k=1}^{\infty}\in l^{p}).$$

In this example we will show that the unilateral weighted backward shift
$$T_{\nu}:(l^{p},d)\rightarrow(l^{p},d)$$
is mixing, where $\nu=(v_{k})_{k=1}^{\infty}$ is a bounded sequence of nonzero scalars with
$$\lim\limits_{k\rightarrow\infty}\prod\limits_{i=1}^{k}|v_{i}|=+\infty.$$

It is obvious that the ordered pair $(l^{p},d)$ is a F-sequence space and ultrabarrelled. Since $(l^{p},d)$ is not locally convex (see \cite[page 91, Example 7.1.1]{Wilansky78}), $(l^{p},d)$ is not a Fr$\acute{e}$chet sequence space or barrelled. Furthermore $(e_{k})_{k=1}^{\infty}$ form a basis for $(l^{p},d)$.
Clearly the unilateral weighted backward shift
$$T_{\nu}:(l^{p},d)\rightarrow(l^{p},d)$$
is continuous. Since $\lim\limits_{k\rightarrow\infty}\prod\limits_{i=1}^{k}|v_{i}|=+\infty$, $(\prod\limits_{i=1}^{k}v_{i})^{-1}e_{k}\rightarrow0$ in $(l^{p},d)$ as $k\rightarrow\infty$. By Theorem 1.6 we have the unilateral weighted backward shift $$T_{\nu}:(l^{p},d)\rightarrow(l^{p},d)$$
is mixing.
\end{example}

\begin{flushright}
  $\Box$
\end{flushright}

In Example 3.4 we show that $(l^{p},d)$ is not a Fr$\acute{e}$chet sequence space or barrelled, so we cannot apply Theorem 1.2 or Theorem 1.4 to the unilateral weighted backward shift $T_{\nu}:(l^{p},d)\rightarrow(l^{p},d)$ to check the mixing behaviour.

\end{document}